\documentclass[letterpaper, 10 pt, conference]{ieeeconf}
\IEEEoverridecommandlockouts
\usepackage{cite}
\usepackage{amsmath,amssymb,amsfonts}
\usepackage{cases}
\allowdisplaybreaks
\usepackage{mathtools,comment}
\usepackage{algorithmic}
\usepackage{graphicx}
\usepackage{tikz,bbm}
\usetikzlibrary{%
    decorations.pathreplacing,%
    decorations.pathmorphing%
}
\usetikzlibrary{arrows.meta}
\usetikzlibrary{automata,positioning}
\usepackage{textcomp}
\usepackage{xcolor}
\let\originalleft\left
\let\originalright\right
\renewcommand{\left}{\mathopen{}\mathclose\bgroup\originalleft}
\renewcommand{\right}{\aftergroup\egroup\originalright}
\newcommand{\bE}{\mathbb{E}}

\newcommand{\cS}{\mathcal{S}}

\newcommand{\cN}{\mathcal{N}}

\newcommand{\cU}{\mathcal{U}}
\newcommand{\ust}{^{\star}}

\newcommand{\bR}{\mathbb{R}}
\newcommand{\bB}{\mathbb{B}}
\newcommand{\bZ}{\mathbb{Z}}

\newcommand{\bP}{\mathbb{P}}
\newcommand{\cE}{\mathcal{E}}

\newcommand{\cB}{\mathcal{B}}

\newcommand{\cL}{\mathcal{L}}
\newcommand{\cY}{\mathcal{Y}}
\newcommand{\cF}{\mathcal{F}}

\newcommand{\Vb}{V^{(\beta)}}
\newcommand{\Qb}{Q^{(\beta)}}
\newcommand{\Jb}{J^{(\beta)}}
\newcommand{\state}{x,\tau,y,b}
\newcommand{\newstate}{x_+,\tau_+,y_+,b_+}
\newcommand{\Tx}{\Tilde{x}}
\newcommand{\Ttau}{\Tilde{\tau}}
\newcommand{\Ty}{\Tilde{y}}
\newcommand{\Tb}{\Tilde{b}}
\newcommand{\TcS}{\Tilde{\cS}}
\newcommand{\TcU}{\Tilde{\cU}}
\newcommand{\Tc}{\Tilde{c}}
\newcommand{\Tstate}{\Tx,\Ttau,\Ty,\Tb}
\newcommand{\Tnewstate}{\Tx_+,\Ttau_+,\Ty_+, \Tb_+}
\newcommand{\Tph}{\Tilde{\phi}}
\newcommand{\Tu}{\Tilde{u}}
\newcommand{\TVb}{\Tilde{V}^{(\beta)}}
\newcommand{\TQb}{\Tilde{Q}^{(\beta)}}
\newcommand{\TJb}{\Tilde{J}^{(\beta)}}
\newcommand{\Tp}{\Tilde{p}}

\newcommand{\itgR}{\int_{\bR}}
\newcommand{\itg}{\int_{\bR_+}}

\DeclareMathOperator*{\argmin}{arg\,min}

\newcommand{\nal}[1]{\begin{align*}#1\end{align*}}
\newcommand{\al}[1]{\begin{align}#1\end{align}}
\newtheorem{assumption}{\textbf{Assumption}}

\newtheorem{definition}{Definition}
\newtheorem{theorem}{Theorem}
\newtheorem{proposition}{Proposition}
\newtheorem{corollary}{Corollary}
\newtheorem{lemma}{Lemma}

\newif\ifuseRomanappendices
\useRomanappendicesfalse

\def\BibTeX{{\rm B\kern-.05em{\sc i\kern-.025em b}\kern-.08em
    T\kern-.1667em\lower.7ex\hbox{E}\kern-.125emX}}

\title{\LARGE \bf
Optimal Scheduling of Uplink-Downlink Networked Control Systems with Energy Harvesting Sensor
}

\author{Manali Dutta and Rahul Singh
\thanks{The authors are with the Department of Electrical Communication Engineering,
Indian Institute of Science, Bengaluru, Karnataka 560012, India (e-mail: manalidutta@iisc.ac.in and rahulsingh@iisc.ac.in)}%
}
\begin{document}

\maketitle

\begin{abstract}
In this work, we consider a wireless networked control system (WNCS) consisting of a plant, a battery-operated sensor, a controller, and an actuator. The battery in the sensor harvests energy from the environment. The sensor then uses this energy for packet transmissions. There are two types of wireless communication channels, (i) sensor--controller channel (also called uplink channel), and (ii) controller--actuator channel (also called downlink channel). The controller is \emph{half-duplex}~\cite{huang2019sense},~\cite{huang2019optimal}, and this prevents it from simultaneously receiving an update from the sensor, and also transmitting a control packet to the actuator. Though frequent transmissions via uplink channel improve controller's estimate of the plant state, thereby improving the quality of control command, but this also reduces the timely control of the plant. Hence, in order to strike a balance between these two competing objectives, we consider the problem of designing an optimal scheduling policy that minimizes the expected cumulative infinite horizon discounted cost, where the instantaneous cost is equal to the square of the plant state. At each time $t$, the scheduler at the sensor has to decide whether it should activate the uplink channel, or downlink. We pose this dynamic optimization problem as a Markov decision process (MDP), in which the state at time $t$ is composed of (i) the plant state $x(t)$, (ii) the age of the data packet available at the controller, denoted by $\tau(t)$, and is equal to the time elapsed since the last successful transmission on the uplink channel, (iii) a binary variable $y(t)$ which indicates the availability of a control packet at the controller, and (iv) the energy level of the battery at the sensor $b(t)$. We show that there exists an optimal scheduling policy that exhibits a threshold structure, meaning that for each time $t$, if there is a control packet available with the controller, then in order to prioritize between the sensor packet and the control packet, the policy considers $|x(t)|$ and the battery level $b(t)$. It then activates the downlink channel in case  $|x(t)|$ exceeds a threshold $x\ust(\tau(t),b(t))$.~To show this result, we first construct an equivalent ``folded MDP''~\cite{chakravorty2018sufficient}, show the structural result of optimal policy hold for this folded MDP, and then unfold this to the original MDP.
\end{abstract}

\begin{keywords}
Wireless networked control system (WNCS), uplink donwlink, scheduling policy, Markov decision process (MDP), threshold structure.
\end{keywords}

\section{Introduction}
\subsection{Literature Review}
A typical wireless networked control system (WNCS) consists of plants, sensors, remote controllers, and actuators that are spatially distributed. These collaborate via a wireless communication network~\cite{zhang2019networked} in order to ensure that the system performance is maximized. The sensors observe the plant, encode the observations into data packets, and then transmit them to the controller via an unreliable wireless communication channel. The controllers estimate the plant state, and then generate control commands based on the information received from the sensors. Subsequently, the controllers transmit these control commands via an unreliable wireless communication channel to the actuators which then utilize these commands to control the plant. With recent developments in wireless networks, cloud computing, edge computing~\cite{wang2023review}, the WNCS has garnered widespread application in diverse fields such as industrial automation~\cite{willig2008recent},~\cite{ascorti2017wireless}, intelligent vehicle~\cite{lin2017integrated},~\cite{li2019adaptive}, intelligent buildings~\cite{pandharipande2014sensor},~\cite{bakule2016decentralized}, and many more~\cite{park2017wireless}.

Wireless transmissions in WNCS can be divided into the following two types: (i) data packet transmissions from the sensor to the controller, and (ii) control packet transmissions from the controller to the actuator. Wireless communication channels are prone to packet losses and delays~\cite{hespanha2007survey}. Since wireless transmitters have limited battery life~\cite{bi2015wireless}, it is not efficient to continually transmit these packets since transmission consumes power. 

There are two types of controllers in the existing literature in NCS,  
(i) those that operate in a \emph{full-duplex} mode, meaning that they can transmit and receive packets simultaneously~\cite{imer2004optimal}~\cite{schenato2007foundations}, 
(ii) \emph{half-duplex} mode, i.e. at any time $t$ the controller can either receive a data packet from the sensor, or it can transmit a control packet to the actuator, but it cannot perform both of these tasks simultaneously\cite{huang2019sense}. The current work assumes that the controller is half-duplex. Henceforth, sensor--controller channel will be denoted uplink channel, while controller--actuator channel will be denoted downlink channel, see Fig.~\ref{fig:setup}. We assume that the sensor is operated by a battery that can store a limited amount of energy. It harvests energy from the environment~\cite{nayyar2013optimal,bhatti2016energy,prauzek2018energy}. The sensor consumes 1 unit of energy to activate the uplink channel.~Now, frequent transmissions via uplink channel improve the quality of the estimates available with the controller, but then this is achieved at the expense of (i) a delay in applying control to the plant since the downlink channel cannot be activated when uplink channel is in use, and (ii) an increased power consumption. Hence, in order to ensure efficient control performance, it is necessary to dynamically schedule transmissions in both these channels. There is an extensive literature that addresses the problem of designing optimal scheduling policies for sensor--controller and controller--actuator channels separately. We will now discuss some of these works in more detail.

\emph{Policies for sensor--controller channel}: The works~\cite{chakravorty2016remote,chakravorty2017structure,ren2017infinite,chakravorty2019remote,wu2019learning,10384144} consider a remote estimation setup, in which the goal is to minimize the estimation error, while keeping the power consumption at a minimal level. No controls are to be applied to the underlying process. 
More specifically, the works~\cite{chakravorty2016remote,chakravorty2017structure,ren2017infinite,chakravorty2019remote} consider the problem of designing a scheduler at the sensor, and an estimator that are jointly optimal.~It is shown that there exists an optimal scheduling policy that has a threshold structure with respect to the current innovation error, which is given by the difference between the current and the \emph{a priori} estimate of the plant state. Thus, the sensor transmits only when the innovation error exceeds a certain threshold. The optimal estimator is shown to be ``Kalman-like,'' i.e., when the estimator receives a data packet, then it updates the estimate of the plant state with the received data, otherwise it predicts the plant state based on the current information available to it.~\cite{chakravorty2016remote} assumes that the packet losses over the channel are i.i.d., and the loss probability is known to the sensor while making decisions.
~\cite{chakravorty2017structure} models the channel as a Gilbert-Elliott channel~\cite{laourine2010betting}, i.e. the state of the channel is described by a two-state Markov chain.~\cite{ren2017infinite},~\cite{chakravorty2019remote} model the channel as a Markov chain with more than two states, but with the difference that~\cite{ren2017infinite} assumes that the channel state is instantaneously known to the sensor while making a decision, while~\cite{chakravorty2019remote} assumes that the sensor knows the channel state with delay of one unit.~\cite{wu2019learning},~\cite{10384144} fix the estimator to be ``Kalman-like,'' and derive optimal scheduling policy.~\cite{wu2019learning} assumes an i.i.d. packet drop channel model with unknown drop probability, and shows that there exists an optimal scheduling policy that is of threshold-type with respect to the time elapsed since the last successful transmission.~\cite{10384144} models the channel as a Gilbert-Elliott channel, and assumes that the channel state is not known to the sensor. It then shows the existence of an optimal threshold-type scheduling policy with respect to the current belief on the channel state, i.e. it transmits only when the conditional probability of the channel, conditioned on the the information available with the scheduler, to be in the good state exceeds certain threshold.

\emph{Policies for controller--actuator channel}: The works~\cite{imer2006optimal,bommannavar2008optimal,shi2012finite,xu2018optimal} consider a WNCS and address the problem of designing optimal controls as well as a scheduling policy that decides when to transmit control packets. They impose constraints on the number of times control can be applied, i.e., controls can be applied only for a $N$ steps out of the total finite time horizon of $T$ steps.~\cite{imer2006optimal} considers a scalar linear quadratic Gaussian (LQG) system. It shows that the optimal control is linear in the optimal estimate of the plant state, the latter quantity can be computed efficiently using the Kalman filter~\cite{kalman1960new}. Moreover, the optimal control schedule is decided by a threshold policy that applies control only when the magnitude of the current plant estimate exceeds a certain threshold.~\cite{bommannavar2008optimal} shows that the results of~\cite{imer2006optimal} continue to hold when the plant has a higher order, and the channel is lossy.~Both~\cite{imer2006optimal}, and~\cite{bommannavar2008optimal} consider cost function that is quadratic in the plant state, and do not impose any control cost. The works~\cite{shi2012finite},~\cite{xu2018optimal} consider a plant which does not have process noise, but the cost function 
is quadratic in the state and controls.~\cite{shi2012finite} designs optimal control schedule for an unstable scalar system. It shows that the optimal schedule is to control the system for the first $N$ time instants, and optimal controls are a linear function of the plant state.~\cite{xu2018optimal} extends the work of~\cite{shi2012finite} to the case when the underlying scalar system is stable, and also allows for multiple systems that share the same channel. It shows that the results of~\cite{shi2012finite} continue to hold even when the plant is stable, as long as certain conditions are satisfied by the system parameters and the cost function.

However, there has been a limited work on designing optimal scheduling policies for the case when the sensor--controller and the controller--actuator channels interfere, which prevents simultaneous transmissions on both these channels~\cite{imer2006measure,huang2019optimal,huang2019sense}.~The work~\cite{imer2006measure} considers a networked LQG system with a perfect communication channel, in which the controller has to decide whether to activate the uplink or the downlink channel. It assumes that the controller knows the state of the plant, and shows that the optimal scheduling policy for the controller is to activate the downlink channel at time $t$ only when the magnitude of the system state at time $t-1$ surpasses a certain threshold. The work~\cite{huang2019optimal} assumes that both the uplink and downlink channels suffer from i.i.d. packet losses, and designs an optimal scheduling policy for the controller. It formulates the optimization problem as a MDP, in which the state at time $t$ is composed of (i) the age of the data packet available with the controller $\tau(t)$, (ii) the age of the latest control packet available with the actuator, denoted $\varphi(t)$.~It shows that the optimal policy is of switching-type, i.e., the controller switches to attempting a transmission on the uplink channel if for each possible age of the control packet $\varphi$, the age of data packet $\tau$ exceeds a certain threshold, and switches to the downlink channel if for each $\tau$ the age of control packet $\varphi$ exceeds a certain threshold.
\subsection{Contributions}
We adopt a half-duplex mode controller proposed in~\cite{huang2019optimal}, in which the controller cannot transmit packets on both the uplink and downlink channels simultaneously. A half-duplex controller is more practical to implement than a full-duplex controller. This is because the latter suffers from self-interference caused due to simultaneous transmissions in both sensor--controller and controller--actuator channels. Several techniques have been developed to mitigate this self-interference, but these techniques result in an increase in device cost and power consumption~\cite{sabharwal2014band}.~However, unlike~\cite{huang2019optimal}, where the controller makes decision regarding which channel should be activated, in our model the sensor makes these decisions at each time $t$.
~The reason why sensor is given the task of making decisions is that as compared with the controller, the sensor has more information available with it since it also knows the current value of the plant state. We assume that the outcome of packet transmissions on both, the uplink as well as the downlink channel, are known at the sensor instantaneously via perfect feedback channels.~The sensor is operated by a battery with a finite energy capacity. Hence, when the battery is empty, uplink channel cannot be activated since activating this channel requires the sensor to spend energy.

~We design an optimal scheduling policy that strikes the right balance between the quality of control commands, and the effective control of the plant. Our contributions are summarized as follows,

1) We formulate the problem faced by the sensor as a dynamic optimization problem, and consider the minimization of the expected value of an infinite horizon discounted cost that is quadratic function of the plant state. In contrast to~\cite{imer2006measure} where the channels are ideal, we assume that wireless channels suffer from packet drops, which are modeled by i.i.d. random variables.

2)
~At each time $t$, the sensor has to choose to activate either uplink or downlink channel, or to keep both the channels idle. We pose this problem as a MDP in which system state at time $t$ is composed of (i) the current value of the plant state $x(t)$, (ii) the age of the data packet available with the controller, denoted by $\tau(t)$, iii) a binary random variable the indicates the presence of a control packet at the controller, and iv) the energy level of the battery. In case the sensor decides that the downlink channel must be activated, then it informs the controller to activate the downlink channel. Since this requires only 1 bit of communication, we assume it is transmitted to the controller instantaneously and without any loss. 

3) Since the instantaneous cost function of the resulting MDP is not bounded, it is immediately not evident whether the value iteration algorithm~\cite{hernandez2012discrete} can be used in order to solve the MDP. However, we show that if certain conditions are satisfied by the MDP, then the value iteration algorithm converges and yields us the value function, which in turn is used to obtain optimal policy.

4) We also derive structural results which characterize an optimal policy. The original MDP is difficult to analyze since the plant state assumes real numbered values. In order to simplify its analysis, we construct a certain ``folded MDP''~\cite{chakravorty2018sufficient} in which the modified plant state takes only non-negative values. We then show the existence of an optimal policy that exhibits a threshold structure with respect to the plant state $x(t)$, i.e., in case there is a packet at the controller, then the scheduler activates the downlink channel only when the current plant state exceeds a certain threshold. This threshold is a function of the age of the data packet available at the controller and the energy level of the battery. 

\textit{Notation}: Let $\bZ_+, \bR, \bR_+, \bR_-$ denote the set of non-negative integers, real numbers, non-negative and negative real numbers, respectively. $\bP(\cdot)$, $\bE(\cdot)$ denote the probability of an event and expectation of a random variable respectively.~$\mathcal{N}(\mu,\sigma^2)$ denotes the Gaussian distribution with mean $\mu$ and variance $\sigma^2$, and $\delta_x(\cdot)$ denotes the delta function with unit mass at $x$.

\section{Problem Formulation}\label{sec:prob_form}
We start by describing the setup in~\ref{sec:sys_mod} and then formulate the problem in~\ref{sec:obj-func}.
\subsection{System Model}\label{sec:sys_mod}
The plant is a discrete-time linear stochastic dynamical system with state at time $t $ denoted by $x(t)$, and evolves as follows, 
\al{
x(t+1) = ax(t) + v(t) + w(t),~t=0,1,2,\ldots,\label{system}
}
where the initial state is $x(0) \sim \cN(0,1)$, $x(t) \in \bR$, $a \in \bR$ describes the dynamics. $v(t) \in \bR$ is the control input that is applied by the actuator at time $t$, and~$w(t) \sim \cN(0,\sigma^2)$ is the plant noise that is assumed to be i.i.d. with distribution $\cN(0,\sigma^2)$. At each time $t$, sensor observes the plant states. It encodes these observations into data packets before transmitting them via an unreliable wireless communication channel to a remote controller. The controller estimates the current state of the plant $x(t)$, and generates controls by utilizing the information available to it until time $t$. The controller encodes the controls into data packets and transmits them to the actuator via another wireless communication channel that connects the controller with the actuator.~This setup is shown in Fig.~\ref{fig:setup}.
\noindent
\begin{figure}[t]
    \centering
    \includegraphics[scale=0.8]{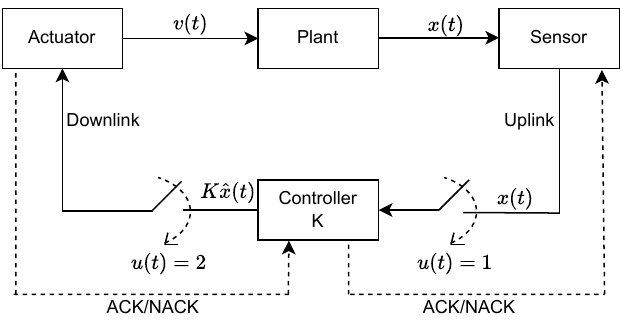}\vspace{-1em}
    \caption{The setup. The plant evolves as $x(t+1)=ax(t)+v(t)+w(t)$ with $w(t) \sim \cN(0,\sigma^2)$, control input $v(t)$, decision variable $u(t) \in \{0,1,2\}$, plant state estimate $\hat{x}(t)$, and controller gain $K$. ACK/NACK indicates the acknowledgment sent by the controller upon receiving a data packet, and by actuator upon receiving a control packet.}
    \label{fig:setup}\vspace{-1em}
\end{figure}
We now describe in detail the channel model, the scheduler which makes transmission decisions for sensor-controller and controller-actuator channels, and the controller.

\textit{Channel Model}: We model the unreliable wireless communication channel as a packet-drop channel with i.i.d. packet drops, whose packet-drop probability is $p \in [0,1]$. 

\textit{Scheduler and controller operations:}  In what follows, we use uplink channel and downlink channel to describe the sensor--controller and controller--actuator channels respectively. We assume that the sensor schedules transmissions on both of these channels. At each time $t$, only one of these two channels can be activated.~The decision of the scheduler at time $t$ is denoted by $u(t)\in \{0,1,2\}$, where, $u(t)=0$ if the sensor does not activate the any of these channels at $t$, $u(t)=1$ if uplink channel is activated at $t$, and $u(t) =2$ if downlink channel is activated.~We assume that the sensor is battery operated. The battery harvests energy from the environment. Each packet transmission by the sensor to the controller on the uplink channel consumes 1 unit of energy.~We use $b(t)$ to denote the energy level of the battery at the sensor at time $t$. The battery can store up to a maximum of $B$ units of energy, i.e., $b(t) \in \bB:=\{0,1,\ldots,B\}$. Clearly, when the battery is empty, i.e. $b(t)=0$, then the sensor cannot transmit, i.e. $u(t)$ cannot assume the value $1$.

At each $t$, the battery harvests a random amount $\ell(t)$ of energy from the environment, where $\ell(t) \in \cL:=\{1,2,\ldots,L\}$. We assume that $\ell(t), t \in \bZ_+$ are i.i.d., and we let $p_{\ell}:=\bP(\ell(t)=\ell)$, where $p_{\ell}\ge 0,\sum_{\ell=0}^L p_{\ell}=1$.~The energy level of the battery evolves as follows,
\al{\label{battery-evolv}
b(t+1)=\min \{b(t)+\ell(t)-\mathbbm{1}(u(t)=1)
, B\},
}
where for an event $\cE$, $\mathbbm{1}(\cE)$ denotes its indicator random variable. Since the information sent by the sensor to the controller for activating the downlink channel requires only 1 bit, we assume that no energy is consumed in this case. 
We assume that the estimator at controller receives the data packet from the sensor with a delay of one time step, so that the controller estimates the plant state at time $t$ by using the data packet generated by the sensor at time $t-1$. We use $\hat{x}(t)$ to denote the estimate of the plant state $x(t)$ at time $t$. It is updated recursively as follows:
\al{\label{x-hat}
\hat{x}(t+1) =
\begin{cases}
    ax(t) &\mbox{ if } u(t) =1, \\
   &\mbox{ and uplink channel is good},\\
    a\hat{x}(t) &\mbox{ otherwise},
\end{cases}
}
with $\hat{x}(0) =0$.
Let $\tau(t)$ denote the age of the data packet available with the controller. This is updated as follows: for $t \in \bZ_+$, we have,
\al{\label{tau}
\tau(t+1) =
\begin{cases}
    1 &\mbox{ if } u(t) =1,\\
    & \mbox{ and uplink channel is good},\\
    \tau(t)+1 &\mbox{ otherwise},
\end{cases}
}
where $\tau(0) = 0$.

We assume that there is perfect feedback on both the channels. Thus, controller sends an acknowledgment to the sensor upon receiving a packet, and similarly the actuator sends an ACK to the controller upon receiving a control packet. Due to this feedback, sensor has knowledge of $\tau(t)$ at time $t$. From~\eqref{x-hat} and~\eqref{tau}, we have,
\al{\label{x-hat-tau}
\hat{x}(t) = a^{\tau(t)} x(t-\tau(t)).
}
Next, when the downlink channel is activated, controller attempts to send a control packet to the actuator. Following a successful transmission, the actuator applies an input to the plant. This control input to the plant at time $t$ is denoted by $v(t)$, and is as follows,
\al{\label{control}
v(t) =
\begin{cases}
    K\hat{x}(t) &\mbox{ if } u(t) =2,\\
    & \mbox{ and downlink channel is good},\\
    0 &\mbox{ otherwise,}
\end{cases}
}
where $K$ is the controller gain. 
\begin{assumption}\label{assump:k}
	We assume that the controller gain satisfies the following~\cite{huang2019sense},~\cite{huang2019optimal},~\cite{o1981discrete},
    \al{
    a+K=0.
}
\end{assumption} 
Under the above assumption, the plant is said to be \emph{one-step controllable}~\cite{huang2019sense},~\cite{huang2019optimal}, which means that in the absence of process noise this controller will drive the plant state to 0.
Upon applying control, actuator discards this control packet.
The actuator sends an acknowledgment to the controller upon receiving a control packet at $t$. Note that the controller does not have any control packet from the time it delivers a control packet to actuator, until the time when sensor delivers a fresh packet to controller. Consequently, the downlink channel must not be activated in this time frame.~We assume that the controller instantaneously notifies the sensor of successful delivery of a control packet. We use $y(t) \in \cY:=\{0,1\}$ to denote whether a control packet is available $(y(t)=1)$ or not $(y(t)=0)$ at the controller at time $t$. Thus, upon delivery of a control packet at time $t$, we set $y(t)$ equal to $0$, while $y(t)$ is set equal to $1$ when a new sensor packet arrives at the controller.~Note that the sensor knows $y(t)$ while making a decision at time $t$. 
\al{\label{y}
y(t)=
\begin{cases}
    0 &\mbox{ if } u(t-1)=2,\\
    & \mbox{ and downlink channel is good},\\
    1 &\mbox{ if } u(t-1)=1,\\
    & \mbox{ and uplink channel is good},\\
    y(t-1) &\mbox{ otherwise.}
\end{cases}
}

Also, since the control input $v(s)=0$ for $s \in \{t-\tau(t), t-\tau(t)+1, \ldots,t-1\}$ under Assumption~\ref{assump:k} and from~\eqref{x-hat-tau},~\eqref{control}, we can write the evolution of $x(t)$~\eqref{system} as follows,
\al{\label{x-t-evolve}
x(t+1)=
\begin{cases}
    & \sum\limits_{s=t-\tau(t)}^{t} w(s) \mbox{ if } u(t)=2,\\
    & \qquad\mbox{ and downlink channel is good},\\
    & ax(t) + w(t) \mbox{ otherwise.}
\end{cases}
}

\subsection{Cost Function}\label{sec:obj-func}
For the model described in Section~\ref{sec:sys_mod}, the goal of the scheduler is to dynamically make decisions $\{u(t)\}_{t \in \bZ_+}$ so as to solve the following problem:
\al{\label{obj-disc}
\min_{\phi} \bE_{\phi} \left[\sum_{t=0}^{\infty} \beta^t x(t)^2\right],
}
where $\beta \in (0,1)$ is the discount factor, $\phi=\{\phi_t\}_{t \in \bZ_+}$ is a measurable policy such that $\phi_t : \cF_t \rightarrow u(t),$ where $\cF_t := \sigma(\{x(s),\tau(s),\hat{x}(s), y(s)\}_{s=0}^{t},\{u(s)\}_{s=0}^{t-1})$ is the information available with the scheduler at time $t$. $\bE_\phi$ is the expectation taken with respect to the measure induced by the policy $\phi.$

\section{MDP formulation}\label{sec:MDP formulation}
In this section, we formulate the optimization problem~\eqref{obj-disc} as a MDP. We then show in Section~\ref{val-iter} that this can be solved using the value iteration algorithm~\cite{hernandez2012discrete}. Section~\ref{folded MDP} introduces the folded MDP, which simplifies the analysis of this MDP, and this allows us to obtain structural results.

The dynamic optimization problem~\eqref{obj-disc} faced by the scheduler can be formulated as the following MDP:
\begin{enumerate}
    \item The state of the MDP at time $t$ is given by $(x(t),\tau(t),y(t),b(t)) \in \bR \times \bZ_+ \times \cY \times \bB$.~This follows from~\eqref{battery-evolv},~\eqref{tau},~\eqref{y}, 
    and~\eqref{x-t-evolve}.~The state-space of the MDP is $\cS:= \bR \times \bZ_+ \times \cY \times \bB$.
    
    \item The decision taken by the scheduler at time $t$ is $u(t) \in \{0,1,2\}$.~$u(t)=0$ denotes that neither uplink nor downlink channel is activated, $u(t)=1$ denotes that the uplink channel is activated at $t$, while $u(t)$ assumes the value $2$ if the downlink channel is activated.~The action-space of the MDP is $\cU:=\{0,1,2\}$. 
    
    \item The instantaneous cost function is $c(x,\tau,y,b,u)=x^2$. 
    \item Let $p(x_+,\tau_+,y_+,b_+ \mid x,\tau,y,b;u)$ denote the transition density function from the current state $(x,\tau,y,b)$, to the state $(x_+,\tau_+,y_+,b_+)$ at next time, when action $u$ is chosen. From~\eqref{tau},
    ~\eqref{x-t-evolve} we have that the transition density $p(x_+,\tau_+, y_+,b_+ \mid x,\tau,y,b;u)$ is as follows,

        i) For $u=0$: We have, 
        \al{
        & p(x_+,\tau_+,y_+,b_+ \mid x,\tau,y,b;0)\notag\\
        & =  p_{\ell} e^{-\frac{(x_+-ax)^2} {2\sigma^2}} \delta_{\min\{b+\ell,B\}}(b_+)  \notag\\
        & \times \delta_{\tau+1}(\tau_+) \delta_{y}(y_+) 
         .\label{df-0}
        }
        ii) For $u=1,b>0$: 
        We have,
        \al{
        & p(x_+,\tau_+,y_+,b_+ \mid x,\tau,y,b;1)\notag\\
        & =  p_{\ell} e^{-\frac{(x_+-ax)^2} {2\sigma^2}} \delta_{\min\{b+\ell-1,B\}}(b_+) \notag\\
        & \times \left(p\delta_{\tau+1}(\tau_+) \delta_{y}(y_+) 
         +(1-p) \delta_1(\tau_+) \delta_1(y_+)\right),\label{df-1}
        }
        iii) For $u=2$: We have the following for $y=1$,
        \al{
        & p(x_+,\tau_+,y_+,b_+ \mid x,\tau,1,b;2)= p_{\ell} \delta_{\min\{b+\ell,B\}}(b_+)\notag\\
        & \times \left(pe^{-\frac{(x_+-ax)^2}{2\sigma^2}} \delta_{\tau+1}(\tau_+) \delta_1(y_+)\right.\notag\\
        & \left. +(1-p)e^{-\frac{{x_+}^2}{2\varepsilon(\tau)}} \delta_{\tau+1}(\tau_+)\delta_0(y_+)\right),\label{df-2}
        }
        where $p$ is the packet drop probability, and 
        \al{
        \varepsilon(\tau) := \frac{1-a^{2(\tau +1)}}{1-a^2}.\label{var}     
        }       
\end{enumerate}

\subsection{Value Iteration}\label{val-iter}
Let $\Vb$ be the value function for the $\beta$-discounted cost problem~\eqref{obj-disc}, i.e., for $(x,\tau,y,b) \in \cS,$
\al{\label{val-func}
\Vb(x,\tau,y,b) = \min_{\phi} \Jb(x,\tau,y,b;\phi),
}
where,
\al{\label{cost-to-go}
\Jb(x,\tau,y,b;\phi) = \bE_{\phi} \left[\sum_{t=0}^{\infty} \beta^t x(t)^2\right].
}
We make the following assumption on $a,p$ which ensures that 
\begin{assumption}\label{assump-stable} 
There exists a scheduling policy $\phi$ under which $\Jb(x,\tau,y,b;\phi)$ is finite for each $(x,\tau,y,b)$.
\end{assumption}
Now, in order to show that the value iteration algorithm can be used to solve~\eqref{obj-disc}, we first need to verify whether our MDP satisfies certain conditions~\cite[p. 46]{hernandez2012discrete}. This is done next. We start with the following definitions. In what follows we denote, $b_+(\ell):=\min\{b+\ell-1,B\}$ and $b'_+(\ell):=\min\{b+\ell,B\}$ where we have suppressed its dependence upon $b$ in order to simplify notation.
\begin{definition}[Transition Law]\label{def:tran-law}
    Let $\mu$ denote the Lebesgue measure on $\bR$. We use $\{P(\cdot \mid x,\tau,y,b,u)\}$ to denote the controlled transition law that describes the transition probabilities. Then, for any Borel measurable subset $\cB$ of $\bR$, $b_+ \in \{b_+(\ell),b'_+(\ell)\}$, and $u \in \cU$, $P$ has a density $p(\cdot \mid \state; u)$~\eqref{df-0}-\eqref{df-2} with respect to $\mu$ which can be written as follows since $\mu$ on $\bR$ is $\sigma-$finite~\cite[Example C.6]{hernandez2012discrete},
    \al{
    & P((\newstate) \in \cB \times \bZ_+ \times \{0,1\} \times \bB \mid \state, u)\notag\\
        & = \sum_{\ell=1}^L\sum _{y_+ \in \cY} \sum_{\tau_+ \in \bZ_+} \notag\\
        & \times \int_{\cB} p(\newstate \mid \state; u)\mu(dx_+)\notag
    }
\end{definition}\vspace{1em}

\begin{definition}[Strongly continuous] 
The transition law $\{P(\cdot \mid x,\tau,y,b,u)\}$ is said to be strongly continuous for each $(x,\tau,y,u,b) \in \cS \times \cU$ if for every bounded measurable function $w: \cS \rightarrow \bR$, the function $w':\cS \times \cU \rightarrow \bR$ is continuous and bounded, where $w'(x,\tau,y,b,u)= \bE[w \mid x,\tau,y,b,u]$.
\end{definition}

\begin{lemma}\label{lemma:VI-conditions}
    Consider the MDP~\eqref{obj-disc}. Under Assumptions~\ref{assump:k} and~\ref{assump-stable}, the following properties hold:
    \begin{itemize}
        \item[{P1)}]  The instantaneous cost $c(x,\tau,y,b,u)=x^2$ is lower semicontinuous\footnote{A function $v: \cS \times \cU  \rightarrow \bR$ is lower semicontinuous on $\cS \times \cU$ if its sublevel sets $\{(x,\tau,y,b,u) \in \cS \times \cU \mid v(x,\tau,y,b,u) \leq z\}$ with $z \in \bR$ are closed in $\bR$.}, non-negative, and inf-compact\footnote{A function $v: \cS \times \cU \rightarrow \bR$ is inf-compact on $\cS \times \cU$ if for every $(x,\tau,y,b) \in \cS$ and $z \in \bR$, the set \{$u \in \cU \mid v(x,\tau,y,b,u) \leq z$\} is compact} on $\cS \times \cU$;
        
        \item[{P2)}]  
        The transition law $\{P(\cdot \mid x,\tau,y,b,u)\}$ is strongly continuous for each $(x,\tau,y,b,u)$;
        
    \end{itemize}
\end{lemma}
\begin{proof}
        {P1)} follows since the instantaneous cost function is continuous, and hence it is lower semicontinuous. Inf-compactness follows since the action set $\cU =\{0,1,2\}$ is finite in our case.
        
        {P2)} We can write $P$ as follows for any Borel measurable subset $\cB$ of $\bR$, $b_+ \in \{b_+(\ell),b'_+(\ell)\}$, and $u \in \cU$,
        \al{
        & P((\newstate) \in \cB \times \bZ_+ \times \{0,1\} \times \bB \mid \state, u)\notag\\
        & =\sum_{\ell=1}^L \sum _{y_+ \in \{0,1\}} \sum_{\tau_+ \in \bZ_+} \notag\\
         & \times \int_{\cB} p(\newstate \mid \state; u)\, dx_+,\notag
        }
        where the equality follows from definition~\ref{def:tran-law} and because $\mu(dx_+) = dx_+$.
        
        Then, P2) follows from the definition of $p$~\eqref{df-0}-\eqref{df-2}~\cite[Example C.6]{hernandez2012discrete}.
\end{proof}
Lemma~\ref{lemma:VI-conditions}, when combined with~\cite[Lemma 4.2.8, Theorem 4.2.3]{hernandez2012discrete} allows us to use value iteration to solve~\eqref{obj-disc}.~This is shown in the following proposition.
It shows that the iterates generated by the value iteration algorithm converge to the true value function $\Vb$.
~It also shows the existence of an optimal deterministic stationary policy (Proposition~\ref{prop:valit}c)). We use $\Vb_n$ to denote the value iterates at stage $n \in \bZ_+$. We start by establishing these iterates, i.e., for all $(x,\tau,y,b) \in \cS$ with $b=0,y=0$,
\al{
        & \Vb_{n+1}(x,\tau,0,0) = \Qb_{n+1}(x,\tau,0,0 ;0),\label{Vn-b0}
        }
        while for $b=0,y=1,$
        \al{
        & \Vb_{n+1}(x,\tau,1,0) = \min_{u \in \{0,2\}} \Qb_{n+1}(x,\tau,1,0 ;u),\label{Vn-b0-y1}
        }
        and for $b>0$,
        \al{
        & \Vb_{n+1}(x,\tau,0,b) = \min_{u \in \{0,1\}}\Qb_{n+1}(x,\tau,0,b ; u),\label{Vn-y0}\\
        & \Vb_{n+1}(x,\tau,1,b) = \min_{u \in \cU} \Qb_{n+1} (x,\tau,1,b;u),\label{Vn-y1}
        }
        where for $u=0$,
        \al{
        & \Qb_{n+1}(x,\tau,y,b;0)= x^2 + \beta \sum_{\ell=1}^L p_{\ell}\notag\\
        & \times \itgR e^{-\frac{(x_+ -ax)^2}{2\sigma^2}} \Vb_n(x_+,\tau+1,y,b'_+(\ell))\,dx_+ \label{Qn0},
        }
        while for $u=1,b>0$,
        \al{
        & \Qb_{n+1}(x,\tau,y,b;1)= x^2 + \beta \sum_{\ell=1}^L p_{\ell}\notag\\
        &  \times \biggl(  p \itgR e^{-\frac{(x_+ -ax)^2}{2\sigma^2}} \Vb_n(x_+,\tau+1,y,b_+(\ell))\,dx_+ \biggr.\notag\\
        & \biggl. + (1-p)\itgR e^{-\frac{(x_+ -ax)^2}{2\sigma^2}} \Vb_n(x_+,1,1,b_+(\ell))\,dx_+\biggr)\label{Qn1},
        }
        and for $u=2,y=1$,
        \al{
        & \Qb_{n+1}(x,\tau,1,b;2) = x^2 + \beta \sum_{\ell=1}^L p_{\ell} \notag\\
        & \times \biggl(p\itgR e^{-\frac{(x_+ -ax)^2}{2\sigma^2}} \Vb_n(x_+,\tau+1,1,b'_+(\ell))\,dx_+ \notag\\
        & +  (1-p)\itgR e^{-\frac{x_+^2}{\varepsilon(\tau)}} \Vb_n(x_+,\tau+1,0,b'_+(\ell))\,dx_+\biggr),\label{Qn2-y1}
        }
        with initialization,
        \al{
        \Vb_0(\state) := 0, \label{v0}
        }

\begin{proposition}\label{prop:valit}
    Consider the MDP~\eqref{obj-disc} with transition density function $p$~\eqref{df-0}-\eqref{df-2}. If Assumptions~\ref{assump:k} and~\ref{assump-stable} hold, then,
    \begin{itemize}
        \item[a)] The value iterates $\Vb_n, n \in \bZ_+$ converges to $\Vb$ i.e., for each $(\state) \in \cS, \lim\limits_{n \rightarrow \infty} \Vb_n(\state) = \Vb(\state)$.
        \item[b)] The value function $\Vb$ satisfies the following optimality equation, for each $(\state) \in \cS$ with $b=0,y=0$,
        \al{
        & \Vb(x,\tau,0,0) = \Qb(x,\tau,0,0 ; 0),\label{V-b0}
        }
        while for $b=0,y=1$,
        \al{
        & \Vb(x,\tau,1,0) = \min_{u \in \{0,2\}}\Qb(x,\tau,1,0 ; u),\label{V-b0-y1}
        }
        and for $b>0$,
        \al{
        & \Vb(x,\tau,0,b) = \min_{u \in \{0,1\}}\Qb(x,\tau,0,b ; u),\label{V-y0}\\
        & \Vb(x,\tau,1,b) = \min_{u \in \cU} \Qb (x,\tau,1,b;u),\label{V-y1}
        }
        where for $b_+ \in \{b_+(\ell),b'_+(\ell)\}$ and from~\eqref{df-0}-\eqref{df-2} we have,
        \al{
        & \Qb(\state;u) = x^2 + \beta \sum_{\ell=1}^L \sum_{y_+ \cY} \sum_{\tau_+ \in \bZ_+} \notag\\
        &  \times \itg \bigl(p(\newstate \mid \state;u) \bigr.\notag\\
        & \hspace{7em}\bigl. \times \Vb(\newstate)\bigr)\,dx_+, \label{Q}
        }
        

        \item[c)] There exists an optimal deterministic stationary policy $\phi\ust$ that minimizes the right-hand side of~\eqref{V-y0} and~\eqref{V-y1} for each $(\state) \in \cS.$
    \end{itemize}
\end{proposition}
\subsection{Folding the MDP}\label{folded MDP}
In this section, we construct a folded MDP~\cite{chakravorty2018sufficient} using the original MDP~\eqref{obj-disc}. The folded MDP has a state-space $\bR_+ \times \bZ_+ \times \{0,1\} \times \bB$, in contrast to the original MDP's state-space $\bR \times \bZ_+ \times \{0,1\} \times \bB$. The folded MDP is easier to analyse since now the plant state assumes only non-negative values, while in the original MDP it takes both negative and non-negative values. We will prove that the folded MDP is equivalent to the original MDP on $\bR_+ \times \bZ_+ \times \{0,1\}\times \bB$ for the purpose of optimizing our cost function. This will allow us to work with the folded MDP instead of the original MDP for subsequent analysis. 
We now show that the value functions $\Vb$~\eqref{V-b0}-\eqref{V-y1} of the original MDP~\eqref{obj-disc} are even with respect to the plant state. This property is instrumental in the construction of the folded MDP.
\begin{proposition}\label{prop:even}
    The functions $\{\Vb(\cdot,\tau,y,b)\}$, $\{\Qb(\cdot,\tau,y,b;u)\}$ are even, i.e., we have the following for all $(\state) \in \cS, u \in \cU$, 
    \nal{
    &\Qb(\state;u)=\Qb(|x|,\tau,y,b;u), \notag\\
    & \Vb(\state)=\Vb(|x|,\tau,y,b).
    }
    Thus, if $\phi\ust(\cdot,\tau,y,b)$ is an optimal policy, then 
    \nal{
    \phi\ust(\state)=\phi\ust(|x|,\tau,y,b).
    }
\end{proposition}
\begin{proof}
        Note that from Proposition~\ref{prop:valit} we have that $\lim_{n \rightarrow \infty} \Vb_n(\state) = \Vb(\state)$ for each $(\state) \in \cS$. Hence it suffices to show that $\Vb_n$ and $\Qb_n, n \in \bZ_+$ are even. We will use induction to show this property for $\Vb_n,\Qb_n,$. The case for $n=0$ is true since $\Vb_0(\state)=0$~\eqref{v0} for each $(\state) \in \cS$. Next, assume that the functions $\Vb_k(\cdot,\tau,y,b), k=1,2,\ldots, n$ are even. We will now show that $\Qb_{n+1}(\cdot,\tau,y,b; u)$ is even for each $(\tau,y,b) \in \bZ_+ \times \{0,1\} \times \bB$ and $u \in \cU$. Consider the following two cases:
        
    Case i): For $u=1, b>0$ we have,
    \al{
    & \Qb_{n+1}(-x,\tau,y,b;1) = x^2 + \beta \sum_{\ell=1}^L p_{\ell}\notag\\
    & \times \biggl( p\itgR e^{-\frac{(x_+ +ax)^2}{2\sigma^2}} \Vb_n(x_+,\tau+1,y,b_+(\ell))\,dx_+ \notag\\
    & + (1-p)\itgR e^{-\frac{(x_+ +ax)^2}{2\sigma^2}} \Vb_n(x_+,1,1,b_+(\ell))\,dx_+ \biggr)\notag\\
    & = x^2 + \beta \sum_{\ell=1}^L p_{\ell}\notag\\
    & \times \biggl(p\itgR e^{-\frac{(-x' +ax)^2}{2\sigma^2}} \Vb_n(-x',\tau+1,y,b_+(\ell))\,dx' \notag\\
    & +  (1-p)\itgR e^{-\frac{(-x' +ax)^2}{2}} \Vb_n(-x',1,1,b_+(\ell))\,dx'\biggr) \notag\\
    & = x^2 + \beta \sum_{\ell=1}^L p_{\ell}\notag\\
    & \times \biggl(p\itgR e^{-\frac{(x' -ax)^2}{2\sigma^2}} \Vb_n(x',\tau+1,y,b_+(\ell))\,dx' \notag\\
    & + (1-p)\itgR e^{-\frac{(x' -ax)^2}{2\sigma^2}} \Vb_n(x',1,1,b_+(\ell))\,dx'\biggr) \notag\\
    & = \Qb_{n+1}(x,\tau,y,b;1),\notag
    }
    where the first equality follows from~\eqref{Qn1}. The second and third equality follows from a change of variables $x_+ = -x'$, and from the induction hypothesis that $\Vb_n(\cdot,\tau,y,b)$ is even, respectively. 
    Hence, $\Qb_{n+1}(\cdot,\tau,y,b;1)$ is even for each $(\tau,y) \in \bZ_+ \times \{0,1\}$ and $b>0$.\\
    Case ii): For $u=2$ we have,
    \al{
    & \Qb_{n+1}(-x,\tau,1,b;2)=x^2 + \beta \sum_{\ell=1}^L p_{\ell}\notag\\
    & \times \biggl( p\itgR e^{-\frac{(x_+ +ax)^2}{2\sigma^2}} \Vb_n(x_+,\tau+1,1,b'_+(\ell))\,dx_+ \notag\\
    & +  (1-p)\itgR e^{-\frac{x_+^2}{\varepsilon(\tau)}} \Vb_n(x_+,\tau+1,0,b'_+(\ell))\,dx_+ \biggr)\notag\\
    & = x^2 + \beta \sum_{\ell=1}^L p_{\ell}\notag\\
    & \times \biggl(p\itgR e^{-\frac{(-x' +ax)^2}{2\sigma^2}} \Vb_n(-x',\tau+1,1,b'_+(\ell))\,dx' \notag\\
    & + (1-p)\itgR e^{-\frac{x'^2}{\varepsilon(\tau)}} \Vb_n(-x',\tau+1,0,b'_+(\ell))\,dx'\biggr) \notag\\
    & =\Qb_{n+1}(x,\tau,1,b;2),\notag
    }
    where the first equality follows from~\eqref{Qn2-y1}. The second equality follows from a change of variables, while the third equality follows from our induction hypothesis that $\Vb_n(\cdot,\tau,1,b)$ is even.~This shows that $\Qb_{n+1}(\cdot,\tau,1,b;2)$ is even.
    
    Similarly, we can show that $\Qb_{n+1}(\cdot,\tau,y,b;0), y \in \cY$ are also even.~Now, it follows from~\eqref{Vn-b0} that $\Vb_{n+1}(\cdot,\tau,0,0)$ is even. 
    $\Vb_{n+1}(\cdot,\tau,1,0)$ is even since it is equal to the pointwise minimum of two even functions $\Qb_{n+1}(\cdot,\tau,1,0;0)$ and $\Qb_{n+1}(\cdot,\tau,1,0;2)$~\eqref{Vn-b0-y1}. Similarly, it follows from~\eqref{Vn-y0} and~\eqref{Vn-y1} that for $b>0$, $\Vb_{n+1}(\cdot,\tau,0,b)$ and $\Vb_{n+1}(\cdot,\tau,1,b)$ are even. 
    ~This completes the induction step and proves our claim. Moreover, from Proposition~\ref{prop:valit}, we have that $\phi\ust(\cdot,\tau,1,0) \in \argmin_{u \in\{0,2\}} \Qb(\cdot,\tau,1,0;u)$, and for $b>0$, $\phi\ust(\cdot,\tau,0,b) \in \argmin_{u \in\{0,1\}} \Qb(\cdot,\tau,0,b;u)$ and $\phi\ust(\cdot,\tau,1,b) \in \argmin_{u \in \cU} \Qb(\cdot,\tau,1,b;u)$. Hence, $\phi\ust(\cdot,\tau,y,b)$ is also even. This completes the proof.         
\end{proof}
We now construct the folded MDP~\cite{chakravorty2018sufficient}. We use $\Tx, \Ttau, \Ty, \Tb, \Tu$ and $\Tph$ to denote the plant state, the age of the data packet available with the controller, the availability of control packet, the energy level of the battery, action, and policy, respectively for the folded MDP.
\begin{definition}[Folded MDP]
    Given the original MDP~\eqref{obj-disc} with transition density function $p$~\eqref{df-0}-\eqref{df-2}, we define the folded MDP with state-space denoted by $\TcS:=\bR_+ \times \bZ_+ \times \{0,1\} \times \bB$, control space $\TcU:=\{0,1,2\}$, and transition density function $\Tp$ as follows,
    \al{\label{df-fold}
    & \Tp(\Tnewstate \mid \Tstate; \Tu) \notag\\
    & = p(\Tnewstate \mid \Tstate; \Tu) \notag\\
    &  + p(-\Tx_+,\Ttau,\Ty_+, \Tb_+ \mid \Tstate; \Tu),
    }
    where $(\Tnewstate),(\Tstate) \in \TcS$, and $\Tu \in \TcU$. The objective function~\eqref{obj-disc} remains the same for the folded MDP, with instantaneous cost given by $\Tc(\Tx,\Ttau,\Ty,\Tb):=\Tx^2$. 
 
\end{definition}
Define,
\al{ 
	&\psi(v,z):=e^{-\frac{v^2}{2z}} ,\notag\\
    &\varphi(v,s,z):=\psi(v-s,z) + \psi(v+s,z).\notag
    }
Next, we can show that under Assumptions~\ref{assump:k} and~\ref{assump-stable}, the folded MDP defined above also satisfies the properties P1)-P3) stated in Lemma~\ref{lemma:VI-conditions}. The proof is similar to Lemma~\ref{lemma:VI-conditions}. Consequently, the conclusions of Proposition~\ref{prop:valit} also hold for the folded MDP, i.e. we can use value iteration to solve the folded MDP, and there exists an optimal stationary deterministic policy. In what follows, we denote $\Tb_+(\ell):=\min\{\Tb+\ell-1,B\}$ and $\Tb'_+(\ell):=\min\{\Tb+\ell,B\}$ where we have suppressed its dependence upon $\Tb$ in order to simplify notation.~Let $\TVb_n$ be the iterates generated during stage $n$ when value iteration is applied to solve the folded MDP. We have for $\Tb=0,\Ty=0$,
\al{
& \TVb_{n+1}(\Tx,\Ttau,0,0)=\TQb_{n+1}(\Tx,\Ttau,0;0)\label{Vn-b0-fold}
}
while for $\Tb=0,\Ty=1$,
\al{
& \TVb_{n+1}(\Tx,\Ttau,1,0)= \min_{\Tu \in \{0,2\}}\TQb_{n+1}(\Tx,\Ttau,y;u)\label{Vn-b0-y1-fold}
}
and for $\Tb>0$,
\al{
& \TVb_{n+1}(\Tx,\Ttau,0,\Tb)= \min_{\Tu \in \{0,1\}} \TQb_{n+1}(\Tx,\Ttau,0,\Tb;\Tu)\label{Vn-y0-fold}\\
& \TVb_{n+1}(\Tx,\Ttau,1,\Tb) = \min_{\Tu \in \TcU} \TQb_{n+1}(\Tx,\Ttau,1,\Tb;\Tu),\label{Vn-y1-fold}
}
where $(\Tstate) \in \TcS$ and $n \in \bZ_+$, and for $\Tu=0$,
\al{
        & \TQb_{n+1}(\Tx,\Ttau,\Ty,\Tb;0) = \Tx^2 + \beta\sum_{\ell=1}^L p_{\ell} \notag\\
        & \times \sum_{\Ty_+ \in \{0,1\}}\sum_{\Ttau_+ \in \bZ_+} \int_{\bR_+} \bigl(\Tp(\Tnewstate \mid \Tx,\Ttau,y,\Tb;0) \bigr.\notag\\
        & \hspace{11em}\bigl. \times \TVb_{n}(\Tnewstate)\bigr)\, d\Tx_+\notag\\
        & = \Tx^2 + \beta\sum_{\ell=1}^L p_{\ell} \notag\\
        & \times \itg \varphi(\Tx_+,a\Tx,\sigma^2) \TVb_n(\Tx_+,\Ttau+1,\Ty,\Tb'_+(\ell))\, d\Tx_+,\label{Qn0-fold}
        }

while for $\Tu=1,\Tb>0$,
\al{
        & \TQb_{n+1}(\Tx,\Ttau,\Ty,\Tb;1) = \Tx^2 + \beta\sum_{\ell=1}^L p_{\ell} \biggl( p\itg \varphi(\Tx_+,a\Tx,\sigma^2)\notag\\
        & \times  \TVb_n(\Tx_+,\Ttau+1,\Ty,\Tb_+(\ell))\, d\Tx_+ +  (1-p) \notag\\
        & \times \itg \varphi(\Tx_+,a\Tx,\sigma^2) \TVb_n(\Tx_+,1,1,\Tb_+(\ell))\, d\Tx_+\biggr), \label{Qn1-fold}
        }
        and for $\Tu=2,\Ty=1$,
        \al{
        & \TQb_{n+1}(\Tx,\Ttau,1;2) =\Tx^2 + \beta\sum_{\ell=1}^L p_{\ell} \biggl( p \itg \varphi(\Tx_+,a\Tx,\sigma^2)\notag\\
        &  \times \TVb_n(\Tx_+,\Ttau+1,1,\Tb'_+(\ell))\, d\Tx_+ +2(1-p)\notag\\
        & \times \itg \psi(\Tx_+,\varepsilon(\Ttau)) \TVb_n(\Tx_+,\Ttau+1,0,\Tb'_+(\ell))\, d\Tx_+\biggr).\label{Qn2-y1-fold}
        }
        In the above, we have,
        \al{
        \TVb_0(\Tstate) := 0, \label{v0-fold}
    }
where~\eqref{Qn0-fold}-\eqref{Qn2-y1-fold} follows from the definition of $\Tp$~\eqref{df-fold}, and from~\eqref{var}, $\varepsilon(\Ttau)=(1-a^{2(\Ttau +1)})/(1-a^2)$.

The following proposition follows from~\cite[Lemma 4.2.8, Theorem 4.2.3]{hernandez2012discrete} and describes the value function $\TVb$ (analogous to \eqref{val-func} for the original MDP~\eqref{obj-disc}), and the convergence of value iterates $\TVb_n$ to $\TVb$ for the folded MDP. These results are analogous to those shown for the original MDP in Proposition~\ref{prop:valit}.
\begin{proposition}\label{prop:valit-fold}
    Consider the folded MDP $(\TcS,\TcU,\Tp,\Tc)$. If Assumptions~\ref{assump:k} and~\ref{assump-stable} hold, then,
    \begin{itemize}
        \item[a)] The value iterates $\TVb_n, n \in \bZ_+$~\eqref{Vn-y0-fold}-\eqref{Vn-y1-fold} converge to $\TVb$ i.e., for each $(\Tstate) \in \TcS$, $\lim\limits_{n \rightarrow \infty} \TVb_n(\state)=\TVb(\Tstate)$. 
        \item[b)] The value function $\TVb$ satisfies the following optimality equation, for each $(\Tstate) \in \TcS$ with $\Tb=0,\Ty=0$,
        \al{
        & \TVb(\Tx,\Ttau,\Ty,0) = \TQb(\Tx,\Ttau,\Ty,0 ; 0),\label{V-b0-fold}
        }
        while for $\Tb=0,\Ty=1$,
        \al{
        & \TVb(\Tx,\Ttau,1,0) = \min_{\Tu \in \{0,1\}}\TQb(\Tx,\Ttau,1,0 ; 0),\label{V-b0-1-fold}
        }
        and for $\Tb>0$,
        \al{
        & \TVb(\Tx,\Ttau,0,\Tb) = \min_{\Tu \in \{0,1\}} \TQb(\Tx,\Ttau,0,\Tb ; \Tu),\label{V-y0-fold}\\
        & \TVb(\Tx,\Ttau,1,\Tb) = \min_{u \in \TcU} \Qb (\Tx,\Ttau,1,\Tb;u),\label{V-y1-fold}
        }
        where,
        \al{
        & \TQb(\Tstate;u) =\Tx^2 + + \beta\sum_{\ell=1}^L \sum_{\Ty_+ \in \{0,1\}} \sum_{\Ttau_+ \in \bZ_+} \notag\\
        &  \times \itg \bigl(\Tp(\Tnewstate \mid \Tstate;u) \bigr.\notag\\
        & \hspace{7em}\bigl. \times \TVb(\Tnewstate)\bigr)\,dx_+. \label{Q-fold}
        }
        \item[c)] There exists an optimal deterministic stationary policy $\Tph\ust$ that minimizes the right-hand side of~\eqref{V-y0-fold} and~\eqref{V-y1-fold} for each $(\Tstate) \in \TcS.$
    \end{itemize}
\end{proposition}
Next, we establish that the folded MDP $(\TcS,\TcU,\Tp,\Tc)$ is equivalent to the original MDP~\eqref{obj-disc}.
\begin{proposition}\label{prop:equiv}
    The functions $\TQb, \TVb, $~\eqref{V-b0-fold}-\eqref{Q-fold} corresponding to the folded MDP agree with $\Qb, \Vb$~\eqref{V-b0}-\eqref{Q} of the original MDP on $\TcS$, i.e., we have the following for each $(\state) \in \cS, u \in \cU,$
    \al{
    & \Qb(\state;u) = \TQb(|x|,\tau,y,b;u),  \notag\\
    & \Vb(\state)=\TVb(|x|,\tau,y,b).
    }
    Thus, if $\Tph\ust$ and $\phi\ust$ is an optimal policy for the folded and original MDP respectively, then
    \al{
    \phi\ust(\state) = \Tph\ust(|x|,\tau,y,b).\notag
    }
\end{proposition}
\begin{proof}
    We will first show that these properties hold for the iterates $\Vb_n~\eqref{Vn-b0}-\eqref{Vn-y1},\TVb_n~\eqref{Vn-b0-fold}-\eqref{Vn-y1-fold},\Qb_n~\eqref{Qn0}-\eqref{Qn2-y1},$ and $\TQb_n~\eqref{Qn0-fold}-\eqref{Qn2-y1-fold}$ using induction. The results would then follow from Propositions~\ref{prop:valit} and~\ref{prop:valit-fold}. Since from Proposition~\ref{prop:even} we have that $\Vb_0(\cdot,\tau,y,b)$ is even, we get $\Vb_0(\state) = \Vb_0(|x|,c,y,b)= \TVb_0(|x|,c,y,b)=0$.
    ~Thus the base case holds. Next, assume that for each $(\state) \in \TcS$, we have $\Vb_k(\state)=\TVb_k(\state)$ for $k=1,2,\ldots,n$. We will show $\Qb_{n+1}(\state;u)=\TQb_{n+1}(\state;u)$. For this purpose, we first consider the case for $b>0$ and $u=1$ for each $(x,\tau,y) \in \bR_+ \times \bZ_+ \times \{0,1\}$,
    \al{
    & \Qb_{n+1}(x,\tau,y,b;1) = x^2 + \beta \sum_{\ell=1}^L p_{\ell}\notag\\
    & \times \biggl( p \itgR \psi(x_+-ax,\sigma^2) V_n(x_+,\tau+1,y,b_+(\ell))\, dx_+\notag\\
    & +  (1-p)\itgR \psi(x_+-ax,\sigma^2) V_n(x_+,1,1,b_+(\ell))\, dx_+ \biggr)\notag\\
    & = x^2 + \beta \sum_{\ell=1}^L p_{\ell} \notag\\
    & \times \biggl[p \biggl(\int_{\bR_+} \psi(x_+-ax,\sigma^2) \Vb_n(x_+,\tau+1,y,b_+(\ell))\, dx_+ \biggr.\notag\\
    & \biggl. + \int_{\bR_-} \psi(x_+-ax,\sigma^2) \Vb_n(x_+,\tau+1,y,b_+(\ell))\, dx_+\biggl)\notag\\
    & +  (1-p) \biggl(\int_{\bR_+} \psi(x_+-ax,\sigma^2) \Vb_n(x_+,1,1,b_+(\ell))\, dx_+ \biggr.\notag\\
    & \biggl. + \int_{\bR_-} \psi(x_+-ax,\sigma^2) \Vb_n(x_+,1,1,b_+(\ell))\, dx_+\biggl)\biggr]\notag\\
    & = x^2 + \sum_{\ell=1}^L p_{\ell} \notag\\
    & \times \biggl[p \biggl(\int_{\bR_+} \psi(x_+-ax,\sigma^2) \Vb_n(x_+,\tau+1,0)\, dx_+ \biggr.\notag\\
    & \biggl. + \int_{\bR_+} \psi(-x_+-ax,\sigma^2) \Vb_n(-x_+,\tau+1,0)\, dx_+\biggl)\notag\\
    & + \beta (1-p) \biggl(\int_{\bR_+} \psi(x_+-ax,\sigma^2) \Vb_n(x_+,1,1,b_+(\ell))\, dx_+ \biggr.\notag\\
    & \biggl. + \int_{\bR_+} \psi(-x_+-ax,\sigma^2) \Vb_n(-x_+,1,1,b_+(\ell))\, dx_+\biggl)\biggr]\notag\\
    & = x^2 + \beta \sum_{\ell=1}^L p_{\ell} \notag\\
    & \times \biggl(p \itgR \varphi(x_+,ax,\sigma^2) \TVb_n(x_+,\tau+1,y,b_+(\ell))\, dx_+\notag\\
    & + (1-p)\itgR \varphi(x_+,ax,\sigma^2) \TVb_n(x_+,1,1,b_+(\ell))\, dx_+\notag\\
    & = \TQb_{n+1}(x,\tau,y,b;1),\notag
    }
    where first equality follows from~\eqref{Qn1}. The fourth equality follows since from Proposition~\ref{prop:even} we have that $\Vb_n(\cdot,\tau,y,b)$ is even and also the induction hypothesis which states that $\Vb_n(x,\tau,y,b)=\TVb_n(x,\tau,y,b)$. Finally, the last equality follows from~\eqref{Qn1-fold}. Hence, from Proposition~\ref{prop:even} combined with the above result, it follows that for each $(x,\tau,y,b) \in \cS$, with $b>0, \Qb_{n+1}(x,\tau,y,b;1)=\Qb_{n+1}(|x|,\tau,y,b;1)=\TQb_{n+1}(|x|,\tau,y,b;1)$. Similarly, it can be shown that that for each $(x,\tau,y,b) \in \cS$, $\Qb_{n+1}(x,\tau,y,b;0)=\Qb_{n+1}(|x|,\tau,y,b;0)=\TQb_{n+1}(|x|,\tau,y,b;0)$. Also, it can be shown that for each $(x,\tau,1,b) \in \cS$, $\Qb_{n+1}(x,\tau,1,b;2)=\Qb_{n+1}(|x|,\tau,1,b;2)=\TQb_{n+1}(|x|,\tau,1,b;2)$.~Thus, the claim for $\Vb_n$ and $\TVb$ follows from Propositions~\ref{prop:valit} and~\ref{prop:valit-fold} respectively. This completes the induction. Moreover, the properties for $\phi\ust$ and~$\Tph\ust$ hold because the optimal decisions for the original MDP and the folded MDP are obtained by minimizing $\Qb(x,\tau,y,b;\cdot),\TQb(x,\tau,y,b;\cdot)$, as is described in Propositions~\ref{prop:valit} and~\ref{prop:valit-fold} respectively. This completes the proof.
\end{proof}

\section{Structural results}\label{sec:structural results}
In this section, we will firstly analyze the folded MDP $(\TcS,\TcU,\Tp,\Tc)$ and derive structural results which characterize optimal scheduling policy of this folded MDP. Then, we unfold the folded MDP back to the original MDP~\eqref{obj-disc} by using Proposition~\ref{prop:equiv}, and this allows us to deduce similar structural results for the original MDP. We begin by demonstrating that the value function $\TVb$~\eqref{V-y0-fold}-\eqref{V-y1-fold} of the folded MDP satisfies a monotonicity property in the following proposition.
\begin{proposition}\label{prop:monotoneV}
    Consider the folded MDP $(\TcS,\TcU,\Tp,\Tc)$ that satisfies Assumptions~\ref{assump:k} and~\ref{assump-stable}. Its value function $\TVb$~\eqref{V-b0-fold}-\eqref{V-y1-fold} satisfies the following properties:
    
    A) For each $(\Ttau,\Ty,\Tb) \in \bZ_+ \times \{0,1\} \times \bB$, the function $\TVb(\cdot,\Ttau,\Ty,\Tb)$ is non-decreasing (with respect to $\Tx$).

    B) For each $(\Tx,\Ttau,\Ty) \in \bR_+ \times \bZ_+ \times \{0,1\}$,  $\TVb(\Tx,\Ttau,\Ty,\cdot)$ are non-increasing (with respect to $\Tb$).
\end{proposition}
\begin{proof}
    A) We will show that this property holds for the iterates $\TVb_n, n \in \bZ_+$ generated by the value iteration~\eqref{Vn-b0-fold}-\eqref{Vn-y1-fold}. The result would then follow from Proposition~\eqref{prop:valit-fold} since $\lim_{n \rightarrow \infty} \TVb_n(\Tstate)=\TVb(\Tstate)$. We will use induction to show this property for $\TVb_n$.~Since $\TVb_0(\Tstate)=0$~\eqref{v0-fold}, the property holds for $n=0$. Next, assume that for each $(\Ttau,\Ty,\Tb) \in \bZ_+ \times \{0,1\} \times \bB$, the functions $\TVb_k(\cdot,\Ttau,\Ty,\Tb), k=1,2,\ldots,n$ are non-decreasing. We will now show that $\TQb_{n+1}(\cdot,\Ttau,\Ty,\Tb;\Tu)$ is also non-decreasing for each $\Tu \in \TcU$. For this purpose, let $\Tx' \geq \Tx, \Tx',\Tx \in \bR_+,$ and consider the following two cases:
    
    Case i): For $\Tu=1,\Tb>0$ we have,
    \al{
    & \TQb_{n+1}(\Tx',\Ttau,\Ty;1) = \Tx'^2 + \beta \sum_{\ell=1}^L p_{\ell} \notag\\
    &  \times \biggl(p \itg \varphi(\Tx_+,a\Tx',\sigma^2)\TVb_n(\Tx_+,\Ttau+1,\Ty,\Tb_+(\ell))\, d\Tx_+\notag\\
    & +  (1-p) \itg \varphi(\Tx_+,a\Tx',\sigma^2)\TVb_n(\Tx_+,1,1,\Tb_+(\ell))\, d\Tx_+\biggr)\notag\\
    & \geq \Tx^2 + + \beta \sum_{\ell=1}^L p_{\ell} \notag\\
    &\times \biggl(p \itg \varphi(\Tx_+,a\Tx,\sigma^2)\TVb_n(\Tx_+,\Ttau+1,\Ty,\Tb_+(\ell))\, d\Tx_+\notag\\
    & +  (1-p) \itg \varphi(\Tx_+,a\Tx,\sigma^2)\TVb_n(\Tx_+,1,1,\Tb_+(\ell))\, d\Tx_+\biggr)\notag\\
    & = \TQb_{n+1}(\Tx,\Ttau,\Ty,\Tb;1),\notag
    }
    where the first equality follows from~\eqref{Qn1-fold}, and the inequality follows from our induction hypothesis on $\TVb_n$ and Lemma~\ref{appen:lemma-inc-func}. 
    
    Case ii): For $\Tu=2$ we have,
    \al{
    & \TQb_{n+1}(\Tx',\Ttau,1,\Tb;2)= \Tx'^2 + \beta \sum_{\ell=1}^L p_{\ell} \notag\\
    &  \times \biggl(p\itg \varphi(\Tx_+,a\Tx',\sigma^2) \TVb_n(\Tx_+,\Ttau+1,1,\Tb'_+(\ell))\, d\Tx_+\notag\\
    & + 2 (1-p)\itg \!\!\psi(\Tx_+,\varepsilon(\Ttau)) \TVb_n(\Tx_+,\Ttau+1,0,\Tb'_+(\ell))\, d\Tx_+\biggr)\notag\\
    & \geq \Tx^2 + \beta \sum_{\ell=1}^L p_{\ell} \notag\\
    & \times \biggl(p\itg \varphi(\Tx_+,a\Tx,\sigma^2) \TVb_n(\Tx_+,\Ttau+1,1,\Tb'_+(\ell))\, d\Tx_+\notag\\
    & + 2 (1-p)\itg \!\!\psi(\Tx_+,\varepsilon(\Ttau)) \TVb_n(\Tx_+,\Ttau+1,0,\Tb'_+(\ell))\, d\Tx_+\biggr)\notag\\
    & = \TQb_{n+1}(\Tx,\Ttau,1,\Tb;2),\notag
    }
    where the first equality follows from~\eqref{Qn2-y1-fold}, and the inequality follows from our induction hypothesis on $\TVb_n$ and Lemma~\ref{appen:lemma-inc-func}.
    
    Similarly, we can that $\TQb_{n+1}(\cdot,\Ttau,\Ty,\Tb;0)$ is non-decreasing. Now, since $\TVb_{n+1}(\Tx,\Ttau,0,0)=\TQb_{n+1}(\Tx,\Ttau,0,0;0)$~\eqref{Vn-b0-fold}, then  $\TVb_{n+1}(\Tx,\Ttau,0,0)$ is non-decreasing. $\TVb_{n+1}(\Tx,\Ttau,1,0)$ is the pointwise minimum of $\TQb_{n+1}(\Tx,\Ttau,1,0;\cdot)$ taken with respect to $\Tu \in \{1,2\}$~\eqref{Vn-b0-y1-fold}, we have that $\TVb_{n+1}(\cdot,\Ttau,0,0)$ is non-decreasing.Similarly, for $\Tb>0$, $\TVb_{n+1}(\Tx,\Ttau,0,b)$~\eqref{Vn-y0-fold} and $\TVb_{n+1}(\Tx,\Ttau,1,b)$~\eqref{Vn-y1-fold} are also non-decreasing.~This completes the induction, and we have proved the claim.

    B) Consider two systems with cost functions $\TJb(\Tx,\Ttau,\Ty,\Tb_1;\Tph_1)$ and $\TJb(\Tx,\Ttau,\Ty,\Tb_2;\Tph_2)$ such that $\Tb_2 \geq \Tb_1$, where $\TJb$ is cost-to-go~\eqref{cost-to-go} for the folded MDP corresponding to the original MDP~\eqref{obj-disc}.~Then, for any sequence of decisions taken under policy $\Tph_1:=(\Tu_0,\Tu_1,\ldots)$ for system with battery level $\Tb_1$, we can always choose a corresponding policy $\Tph'_2:=(\Tu'_0,\Tu'_1,\ldots)$ for the system with battery level $\Tb_2$ such that  $\TJb(\Tx,\Ttau,\Ty,\Tb_1;\Tph_1)=\TJb(\Tx,\Ttau,\Ty,\Tb_2;\Tph'_2)$. This means that $\TJb(\Tx,\Ttau,\Ty,\Tb_2;\Tph_1) \le \TJb(\Tx,\Ttau,\Ty,\Tb_1;\Tph_1)$ because $\Tb_2 \geq \Tb_1$. Thus, $\TVb(\Tx,\Ttau,\Ty,\Tb_2) \leq \TVb(\Tx,\Ttau,\Ty,\Tb_1)$. This completes the proof. 
\end{proof}    
Next, we introduce the class of threshold-type policies for the folded MDP as well as the original MDP. We will then show that an optimal scheduling policy for the folded MDP and the original MDP belongs to these classes.
\begin{definition}[Threshold-type Policy]
    We say that a scheduling policy $\Tph: \TcS \rightarrow \TcU$ for the folded MDP is of threshold-type if 
    it activates the downlink channel only when the plant state exceeds a certain threshold, that depends upon the age of the packet available with the controller and the battery level. Thus, the set of states in which the downlink channel is activated are given as follows,
    \al{
    \left\{ (\Tx,\Ttau,1,\Tb): \Tx \ge  \Tx\ust(\Ttau,\Tb) \right\},
    }
where $\Tx\ust(\Ttau)$ is the threshold value corresponding to the age of the packet at the controller equal to $\Ttau$.   

A threshold-type scheduling for the original MDP~\eqref{obj-disc} is similarly defined, and activates the downlink channel only in the following set of states,
\al{
\left\{ (x,\tau,1,b): |x| \geq x\ust(\tau,b)\right\}.
}
\end{definition}
We now show the existence of an optimal threshold-type scheduling policy for the folded MDP. 
\begin{theorem}\label{main-theorem}
    Consider the folded MDP $(\TcS,\TcU,\Tp,\Tc)$ satisfying Assumptions~\ref{assump:k} and~\ref{assump-stable}.~For each $\Ttau \in \bZ_+$, and $\Tb \in \bB$, there exists a threshold $\Tx\ust(\Ttau,\Tb)$ such that for states of the form $(\Tx,\Ttau,1,\Tb)$, it is optimal to activate the downlink channel only when we have $\Tx \geq \Tx\ust(\Ttau,\Tb)$. Thus, there exists an optimal scheduling policy that admits a threshold structure.
\end{theorem}
\begin{proof}
From~\eqref{V-y1-fold} we have that in order to prove the claim, it is sufficient to show that the following two conditions hold,
    
    C1) $G(\Tx):=\TQb(\Tx,\Ttau,1,\Tb;2)-\TQb(\Tx,\Ttau,1,\Tb;1)$ is non-increasing in $\Tx$. We use $G_{\Tx}$ to denote $\frac{\partial G}{\partial \Tx}$. Then,
    \al{
    & \frac{\partial G}{\partial \Tx}= \beta \sum_{\ell=1}^L p_{\ell} \biggl[p\frac{\partial}{\partial \Tx} \biggl(\itg \varphi(\Tx_+,a\Tx,\sigma^2)\notag\\
    &  \bigl(\TVb_n(\Tx_+,\Ttau+1,1,\Tb'_+(\ell))-\TVb_n(\Tx_+,\Ttau+1,1,\Tb_+(\ell))\bigr)\, d\Tx_+\biggr)\notag\\
    & - (1-p) \frac{\partial}{\partial \Tx} \itg \varphi(\Tx_+,a\Tx,\sigma^2) \TVb_n(\Tx_+,1,1,\Tb_+(\ell))\, d\Tx_+\biggr]\notag\\
    & \le - (1-p) \frac{\partial}{\partial \Tx} \itg \varphi(\Tx_+,a\Tx,\sigma^2) \TVb_n(\Tx_+,1,1,\Tb_+(\ell))\, d\Tx_+\notag\\
    & \le 0,\notag
    }
    where the first inequality follows from Proposition~\ref{prop:monotoneV} since $\Tb'_+(\ell) \geq \Tb_+(\ell)$ for each $l \in \{1,2,\ldots, L\}$, and the second inequality follows from Proposition~\ref{prop:monotoneV} and Lemma~\ref{appen:lemma-inc-func} in the appendix. This implies that for $\Tx' \geq \Tx$, if $G(\Tx) \le 0$, then $G(\Tx') \le G(\Tx) \leq 0$.

    C2) For $\Tx' \ge \Tx$, if $\TQb(\Tx,\Ttau,1,\Tb;2)-\TQb(\Tx,\Ttau,1,\Tb;0) \le 0$, then $\TQb(\Tx',\Ttau,1,\Tb;2)-\TQb(\Tx',\Ttau,1,\Tb;0) \le 0$. Consider,
    \al{
    & \TQb(\Tx',\Ttau,1,\Tb;2)-\TQb(\Tx',\Ttau,1,\Tb;0) =\beta\sum_{\ell=1}^L p_{\ell}\notag\\
    & \times \bigg[(p-1) \itg \varphi(\Tx_+,a\Tx',\sigma^2)
    \TVb_n(\Tx_+,\Ttau+1,1,\Tb'_+(\ell))\, d\Tx_+ \notag\\
    & +2(1-p) \itg \psi(\Tx_+,\varepsilon(\Ttau)) \TVb_n(\Tx_+,\Ttau+1,0,\Tb'_+(\ell))\, d\Tx_+\biggr]\notag\\
    & \leq \TQb(\Tx,\Ttau,1,\Tb;2)-\TQb(\Tx,\Ttau,1,\Tb;0)\notag\\
    & \leq 0,\notag
    }
    where the first inequality follows from Proposition~\ref{prop:monotoneV} and Lemma~\ref{appen:lemma-inc-func} in the appendix. This completes the proof.
\end{proof}
We will now unfold the folded MDP $(\TcS,\TcU,\Tp,\Tc)$ in order to obtain similar results for the original MDP~\eqref{obj-disc}. The following corollary shows structural results for optimal policy of the original MDP.
\begin{corollary}
    Consider the original MDP~\eqref{obj-disc} satisfying Assumptions~\ref{assump:k} and~\ref{assump-stable}. Then, there exists an optimal scheduling policy $\phi\ust$ corresponding to $\Vb$~\eqref{V-b0}-\eqref{V-y1} that is of threshold-type.
\end{corollary}
\begin{proof}
    We have that $\TVb(|x|,\tau,y,b)$ is non-decreasing in $|x|$ from Lemma~\ref{appen:V-inc}. The result then follows from Proposition~\ref{prop:equiv} and Theorem~\ref{main-theorem}.
\end{proof}

\section{Conclusion}
We analyze a WNCS that consists of a plant, a battery-operated sensor, a controller, and an actuator. The sensor has to dynamically decide whether to activate the uplink (sensor to controller) channel or downlink (controller to actuator) channel, or to remain inactive so as to minimize the expected value of an infinite horizon discounted cost that is a quadratic function of the plant state. The wireless uplink and downlink channels are modeled as i.i.d. packet drop channels. We pose this problem as a MDP and show that the system state is given by a four tuple consisting of the plant state, age of the sensor packet available with the controller, a binary variable which indicates whether or not a control packet is available at the controller, and the energy level of the battery. The original MDP is not amenable to analysis since the plant state can assume positive as well as non-positive values, hence we construct a simpler folded MDP, and show that these two are equivalent MDPs. It is shown that the folded MDP admits an optimal scheduling policy with a threshold structure, i.e. in case there is a control packet, the scheduler activates the downlink channel only when the magnitude of the plant state exceeds a certain threshold, where the threshold depends upon the age of the packet and the battery energy level. Upon unfolding the MDP, we recover a similar structural result for the original MDP.~This work can be extended in multiple directions. Firstly, since the state space is infinite, it is not computationally feasible to use the value iteration algorithm to obtain these thresholds. Hence, we would like to design an algorithm that performs an efficient search over the policy space in order to yield these optimal thresholds.~Finally, we would also like to consider the case when the plant is described by a vector process, instead of a scalar. 
\appendix

\section{Preliminary Lemma}\label{lemma:inc-func}
\renewcommand{\thesection}{\Alph{section}}
For ease of reference, we restate the notation here:
\al{ &\psi(v,z):=e^{-\frac{v^2}{2z}},\notag\\
& \varphi(v,s,z):=\psi(v-s,z) + \psi(v+s,z).\notag}
\begin{lemma}\label{appen:V-inc}
    Consider the original MDP~\eqref{obj-disc} satisfying Assumptions~\ref{assump:k} and~\ref{assump-stable}. For each $(\tau,y,b) \in \bZ_+ \times \{0,1\} \times \bB$, the value iterates $\Vb(x,\tau,y,b)$~\eqref{Vn-b0}-\eqref{Vn-y1},  are non-decreasing in $|x|.$
\end{lemma}
\begin{proof}
    We have that for $(\state) \in \cS$, $\Vb(\state)=\Vb(|x|,\tau,y,b)=\TVb(|x|,\tau,y,b)$ from Proposition~\ref{prop:equiv}. The result then follows from Proposition~\ref{prop:monotoneV}.
\end{proof}
\begin{lemma}\label{appen:lemma-inc-func}
    Consider the folded MDP $(\TcS,\TcU,\Tp,\Tc)$. Let $\Tx' \geq \Tx$ where $\Tx',\Tx \in \bR_+$. Assume for each $(\Ttau,\Ty,\Tb) \in \bZ_+ \times \{0,1\} \times \bB$, the value iterate corresponding to step $n,$ $\TVb_n(\cdot,\Ttau,\Ty,\Tb)$ is non-decreasing. Then, $\TVb_n$ satisfies the following,
    \al{
    & \itg \varphi(\Tx_+,a\Tx',\sigma^2)\TVb_n(\Tx_+,\Ttau_+,\Ty_+,\Tb_+)\, d\Tx_+ \notag\\
    & \geq \itg \varphi(\Tx_+,a\Tx,\sigma^2)\TVb_n(\Tx_+,\Ttau_+,\Ty_+,\Tb_+)\, d\Tx_+\notag
    }
\end{lemma}
\begin{proof}
    For $\Bar{x} >0,$ consider the following,
    \al{
    S(\Bar{x},\Tx):= \int_{\Bar{x}}^{\infty} \varphi(\Tx_+,a\Tx,\sigma^2)\,d\Tx_+.\label{S}
    }
    We will now show that~\eqref{S} is non-decreasing in $\Tx$ for both $a \geq 0$ and $a <0.$ For this, denote $S_{\Tx}(\Bar{x},\Tx):= {\partial S}/{\partial \Tx}.
    $ Then,
    \al{\label{dS}
    S_{\Tx}(\Bar{x},\Tx)=a\varphi(\Bar{x},a\Tx,\sigma^2) \geq 0,
    }
    where~\eqref{dS} follows because for $a \geq 0, \varphi(\Bar{x},a\Tx,\sigma^2) \geq 0$ and for $a <0,\varphi(\Bar{x},a\Tx,\sigma^2)<0.$\par
    The claim then follows from~\cite[Lemma 4.7.2, p.106]{Puterman2014markov} since $S(\Bar{x},\Tx') \geq S(\Bar{x},\Tx)$ with $S(0,\Tx')=S(0,\Tx),$ and $\TVb$ is non-decreasing in $\Tx.$
\end{proof}

\bibliographystyle{IEEEtran}
\bibliography{refs}

\end{document}